\documentclass{birkjour}
\usepackage{graphics}
\usepackage{epsfig}
\usepackage{amsfonts}
\usepackage{amssymb}
\usepackage{amsthm}
\usepackage{newlfont}
\usepackage{slashbox}
 \newtheorem{thm}{Theorem}[section]

 \newtheorem{prop}[thm]{Proposition}
 \theoremstyle{definition}
 \newtheorem{defn}[thm]{Definition}
 \theoremstyle{remark}

 \numberwithin{equation}{section}

\begin{document}

%-------------------------------------------------------------------------
% editorial commands: to be inserted by the editorial office
%
%\firstpage{1} \volume{228} \Copyrightyear{2004} \DOI{003-0001}
%\seriesextraline{This is the Concrete Title of this Book\br H.E. R and S.T.C. W, Eds.}
%\firstpage{1}

\title[The $M$-Polynomial and Topological Indices of Generalized M\"{o}bius Ladder and Its Line Graph]
 {\begin{center} The $M$-Polynomial and Topological Indices of Generalized M\"{o}bius Ladder and Its Line Graph \end{center}}

%----------Author 1
\author{Abdul Rauf Nizami}
\address{Abdus Salam School of Mathematical Sciences, GC University, Lahore-Pakistan}
\email{arnizami@sms.edu.pk}
%----------Author 2
\author{Muhammad Idrees}
\address{School of Automation, Beijing Institute of Technology, Beijing-China}
\email{idrees@bit.edu.cn}
%----------Author 3
\author{Numan Amin}
\address{Abdus Salam School of Mathematical Sciences, GC University, Lahore-Pakistan}
\email{numan.amin@sms.edu.pk}
%----------Author 4

%\thanks{This research is partially supported by Higher Education Commission of Pakistan.}


\maketitle
%%% ----------------------------------------------------------------------
%%\begin{center}Abdul Rauf NIZAMI\\
%%$^{*}$Division of Science and Technology\\ University of Education, Lahore-Pakistan\\
%%arnizami@ue.edu.pk\\
%%\vspace{5pt}
%%Mobeen MUNIR$^{*}$\\
%%mobeenmunir@gmail.com
%%\end{center}

%\vspace{20pt}
%\date{today}%February 08, 2013
\begin{abstract}
The $M$-polynomial was introduced by Deutsch and Klav\v{z}ar in 2015 as a graph polynomial to provide an easy way to find closed formulas of degree-based topological indices, which are used to predict physical, chemical, and pharmacological properties of organic molecules. In this paper we give general closed forms of the $M$-polynomial of the generalized M\"{o}bius ladder and its line graph. We also compute Zagreb Indices, generalized Randi\'{c} indices, and symmetric division index of these graphs via the $M$-polynomial.

\end{abstract}
%%%---------------------------------------------------------------------

%%%------------------------------------------------------------------------
%%% ----------------------------------------------------------------------
%\maketitle

\pagestyle{myheadings}
\markboth{\centerline {\scriptsize
 Nizami, Idrees, and Numan}} {\centerline {\scriptsize
 The $M$-Polynomial of Generalized M\"{o}bius Ladder and Its Line Graph}}
%%% ----------------------------------------------------------------------
%\tableofcontents
%%----------------------------------------------- section --------------------------------------------------------
%%----------------------------------------------- section --------------------------------------------------------
%\tableofcontents
\section{Introduction} The $M$-polynomial was introduced by E. Deutsch and S. Klav\v{z}ar in 2015 in \cite{Deutsch:15} as a graph invariant to play a role for degree-based invariants parallel to the role the Hosoya polynomial plays for distance-based invariants. \\

\noindent The $M$-polynomial has applications in mathematical chemistry and pharmacology. The most interesting application of the M-polynomial is that almost all degree-based graph invariants, which are used to predict physical, chemical, and pharmacological properties of organic molecules, can be recovered from it; for more information please see \cite{Das-Xu-Nam:15,Gutman:13,Gutman-Tosovic:13,Khalifeha-Yousefi-Ashrafi:09,Xu-Tang-Liu-Wang:15,Zhou-Gutman2:05,Zhou:07}.\\

\noindent The $M$-polynomial and related topological indices have been studied for several classes of graphs. In 2015 Deutsch and  Klav\v{z}ar gave M-polynomial, first Zagreb, and second Zagreb indices of polyomino chains, starlike trees, and triangulenes \cite{Deutsch:15}. In 2017 Mobeen \emph{et al} gave $M$-polynomial and several degree-based topological indices of titania nanotubes \cite{mobeen:17}, of triangular Boron Nanotubes \cite{mobeen-waqas-nizami2:17}, and of jahangir graph \cite{mobeen-waqas-nizami3:17}. In 2017 Ajmal \emph{et al} gave M-polynomial and topological indices of generalized prism network \cite{ajmal:17}.\\

\noindent Several degree-based topological indices, which play important role in mathematical chemistry, can be recovered from the $M$-polynomial: The most famous degree-based index is the Randi\'{c} index and was introduced by Milan Randi\'{c} in 1975 \cite{Randic:75}. It is often used in cheminformatics for investigations of organic compounds; for more information, please see \cite{Gutman-Furtula:08,Li-Gutman:06,Todeschini-Consonni:00}. Later in 1998, working independently, Amic \emph{et al} \cite{Amic-Beslo:98} and Bollobas-Erdos \cite{Bollobas-Erdos:98} proposed the generalized Randi\'{c} index; for more information, please see \cite{Li-Gutman2:06,Hu-Li-Shi-Xu-Gutman:05}. Gutman and Trinajsti\'{c} introduced first Zagreb and second Zagreb indices in 1972 \cite{Gutman-Trinajstic:72}. The augmented Zagreb index was proposed by Furtula \emph{et al.} in 2010 in \cite{Furtula-Graovac-Vukicevic:10} and is useful for computing heat of information of alkanes \cite{Gutman-Furtula:08,Huang-Liu-Gan:12}. To know more about topological indices, their computing, and their applications we refer the reader to \cite{Braun-Kerber-Meringer-Rucker:05,Deutsch:15,Fajtlowicz:87,Gutman-Das:04,mobeen-waqas-nizami2:17,
mobeen-waqas-nizami3:17,Nikolic-Kovacevic-Milicevic-Trinajstic:08,Xueliang-Yongtang:08,Zhou-Gutman2:05,Zhou2:07}.\\

%For Most people understand that the harmonic index appeared in \cite{Fajtlowicz:87}.\\
% SDD is a significant predictor of total surface area for polychlorobiphenyls

\noindent This article is devoted to study the $M$-polynomial. We not only give the general forms of the $M$-polynomials of the generalized M\"{o}bius ladder and its line graph but also recover first Zagreb, second Zagreb, second modified Zagreb, generalized Randi\'{c}, reciprocal generalized Randi\'{c}, and symmetric division indices from them.

\section{Preliminary Notes}
This section covers the definitions of graph, degree of a vertex, line graph, molecular graph, $M$-polynomial, topological index, Zagreb indices, generalized Randi\'{c} indices, and generalized M\"{o}bius ladder.\\

\noindent A \emph{graph} $G$ is a pair $(V,E)$, where $V$ is the set of vertices and $E$ the set of edges. The edge $e$ between two vertices $u$ and $v$ is denoted by $(u,v)$. The \emph{degree} of a vertex $u$, denoted by $d_{u}$ is the number of edges incident to it. A \emph{path} from a vertex $v$ to a vertex $w$ is a sequence of vertices and edges that starts from $v$ and stops at $w$. The number of edges in a path is the \emph{length} of that path. A graph is said to be \emph{connected} if there is a path between any two of its vertices.
\begin{center}
\begin{minipage}{10cm}
\centering  \epsfig{figure=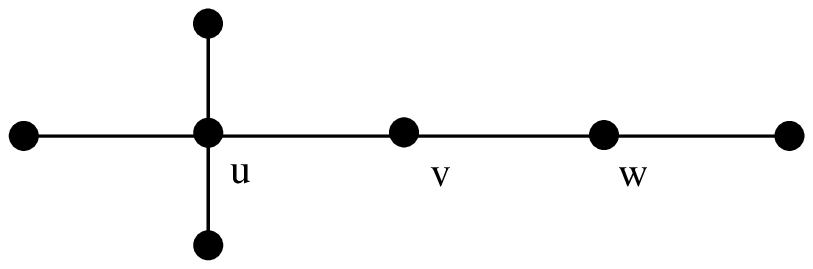,height=2.0cm}
\par \tiny{Figure 1: A connected graph with $d_{u}=4$ to $d_{v}=2$}
\end{minipage}
\end{center}
\noindent The \emph{line graph} of a graph $G$, written $L(G)$, is the graph whose vertices are edges of $G$, and when $e=(u,v)$ and $e'=(v,w)$ are adjacent edges of $G$ then $(e,e')$ is an edge of $L(G)$.
\begin{center}
\begin{minipage}{10cm}
\centering  \epsfig{figure=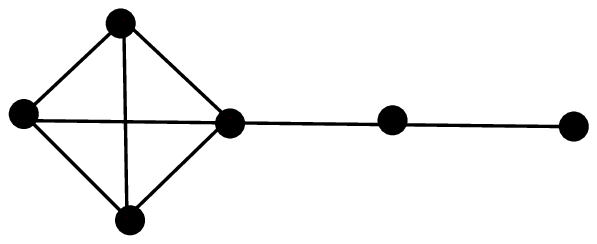,height=2.0cm}
\par \tiny{Figure 2: The line graph of the above graph}
\end{minipage}
\end{center}
\noindent A \emph{molecular graph} is a representation of a chemical compound in terms of graph theory. Specifically, molecular graph is a graph whose vertices correspond to (carbon) atoms of the compound and whose edges correspond to chemical bonds. For instance, Figure 1 represents the molecular graph of 1-bromopropyne ($CH_{3}-C\equiv C-Br$). \\
%The graph given in Definition 2.1 corresponds to the molecule $CH_{3} - S - S - H$.

\noindent In the following by $G$ we shall mean a connected graph, $E$ its edge set, $V$ its vertex set, and $e=(u,v)$ its edge joining the vertices $u$ and $v$.
\begin{defn}\cite{Deutsch:15}
The M-polynomial of $G$ is
$$M(G;x,y)=\sum_{i\leq j}m_{ij}x^{i}y^{j},$$
where $m_{ij}$ is the number of edges $e=(u,v)$ of $G$ with $d_{u}=i$ and $d_{v}=j$.
\end{defn}
\begin{defn}
\noindent A function $I$ which assigns to every connected graph $G$ a unique number $I(G)$ is called a \emph{graph invariant}. Instead of the function $I$ it is custom to say the number $I(G)$ as the invariant. An invariant of a molecular graph which can be used to determine structure-property or structure-activity correlation is called the \emph{topological index}. A topological index is said to be\emph{ degree-based} if it depends on degrees of the vertices of the graph.
\end{defn}
\noindent The following are definitions of some those degree-based indices that have connection with the $M$-polynomial.\\

\noindent The \emph{first Zagreb, second Zagreb}, and \emph{second modified Zagreb} indices of $G$ are respectively
$M_{1}(G)=\sum_{e\in E}(d_{u}+d_{v})$, $M_{2}(G)=\sum_{e\in E}d_{u}\times d_{v}$, and $MM_{2}(G)=\sum_{e\in E}\frac{1}{d_{u}\times d_{v}}$.\\

\noindent The \emph{generalized Randi\'{c}} and \emph{reciprocal generalized Randi\'{c}} indices of $G$ are respectively
$R_{\alpha}(G)=\sum_{e\in E}(d_{u}\times d_{v})^{\alpha}$ and $RR_{\alpha}(G)=\sum_{e\in E}\frac{1}{(d_{u}\times d_{v})^{\alpha}}$.\\

\noindent The \emph{symmetric division index} of $G$ is
$$SSD(G)=\sum_{e\in E}\Big\{\frac{\min(d_{u},d_{v})}{\max(d_{u},d_{v})}+\frac{\max(d_{u},d_{v})}{\min(d_{u},d_{v})}\Big\}.$$

\noindent A remarkable property of the $M$-polynomial is that all the above degree-based indices can be recovered from it, using the relations given in the following table.

\begin{center}
\begin{tabular}{|l|l|}
  \hline
  Index &  Relation with the $M$-polynomial \\ \cline{1-2}
  First Zagreb & $M_{1}(G)=(D_{x}+D_{y})(M(G))_{x=y=1}$ \\
  Second Zagreb & $M_{2}(G)=(D_{x}(\times D_{y})(M(G))_{x=y=1}$ \\
  Modified Second Zagreb & $MM_{2}(G)=(S_{x}\times S_{y})(M(G))_{x=y=1}$ \\
  Generalized Randi\'{c} & $R_{\alpha}(G)=(D_{x}^{\alpha}\times D_{y}^{\alpha})(M(G))_{x=y=1}$ \\
  Reciprocal Generalized Randi\'{c} & $RR_{\alpha}(G)=(S_{x}^{\alpha}\times S_{y}^{\alpha})(M(G))_{x=y=1}$\\
  Symmetric Division & $SDD(G)=(D_{x}S_{y}+ D_{y}S_{x})(M(G))_{x=y=1}$\\ \cline{1-2}
   & \\
   & $D_{x}=x\frac{\partial}{\partial x} M(G)$, $D_{y}=y\frac{\partial}{\partial y} M(G)$\\
   Where & $S_{x}=\int_{0}^{x}\frac{M(G;t,y)}{t}dt$, $S_{y}=\int_{0}^{y}\frac{M(G;x,t)}{t}dt$\\
   & \\
    \hline
\end{tabular}
\end{center}

\begin{defn}\cite{Hongbin-Idrees-Nizami-Munir:17} Consider the Cartesian product $P_{m}\times P_{n}$ of paths $P_{m}$ and $P_{n}$ with vertices $u_{1},u_{2},\ldots,u_{m}$ and $v_{1},v_{2},\ldots,u_{n}$, respectively. Take a $180^{o}$ twist and identify the vertices $(u_{1},v_{1}),(u_{1},v_{2}),\ldots,(u_{1},v_{n})$ with the vertices $(u_{m},v_{n})$, $(u_{m},v_{n-1})$, $\ldots,(u_{m},v_{1})$, respectively, and identify the edge $\big((u_{1},i)$, $(u_{1},i+1)\big)$ with the edge $\big((u_{m},v_{n+1-i})$, $(u_{m},v_{n-i})\big)$, where $1\leq i\leq n-1$. What we receive is the \emph{generalized M\"{o}bius ladder} $M_{m,n}$.
\end{defn}

\noindent You can see $M_{7,3}$ in the following figure.
 \begin{center}
\begin{minipage}{10cm}
\centering  \epsfig{figure=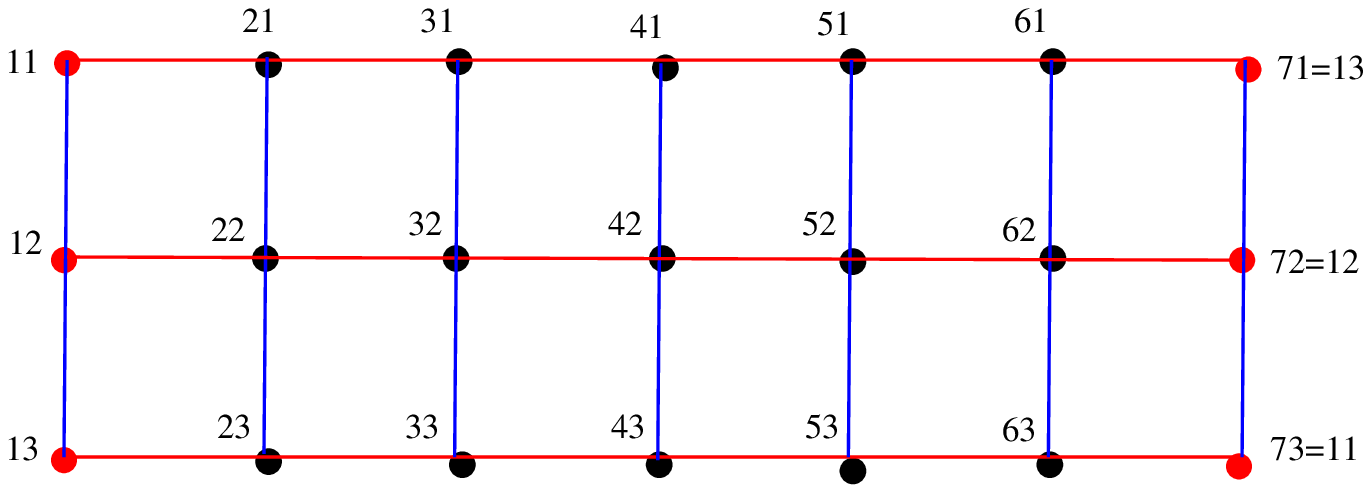, height=3.0cm}
\par \tiny {Figure 3: The grid form of the generalized M\"{o}bius ladder $M_{7,3}$ }%$P_{7}\times P_{3}$ with complete simple labels
\end{minipage}
\end{center}

The original form of $M_{7,3}$ is:
\begin{center}
\begin{minipage}{10cm}
\centering  \epsfig{figure=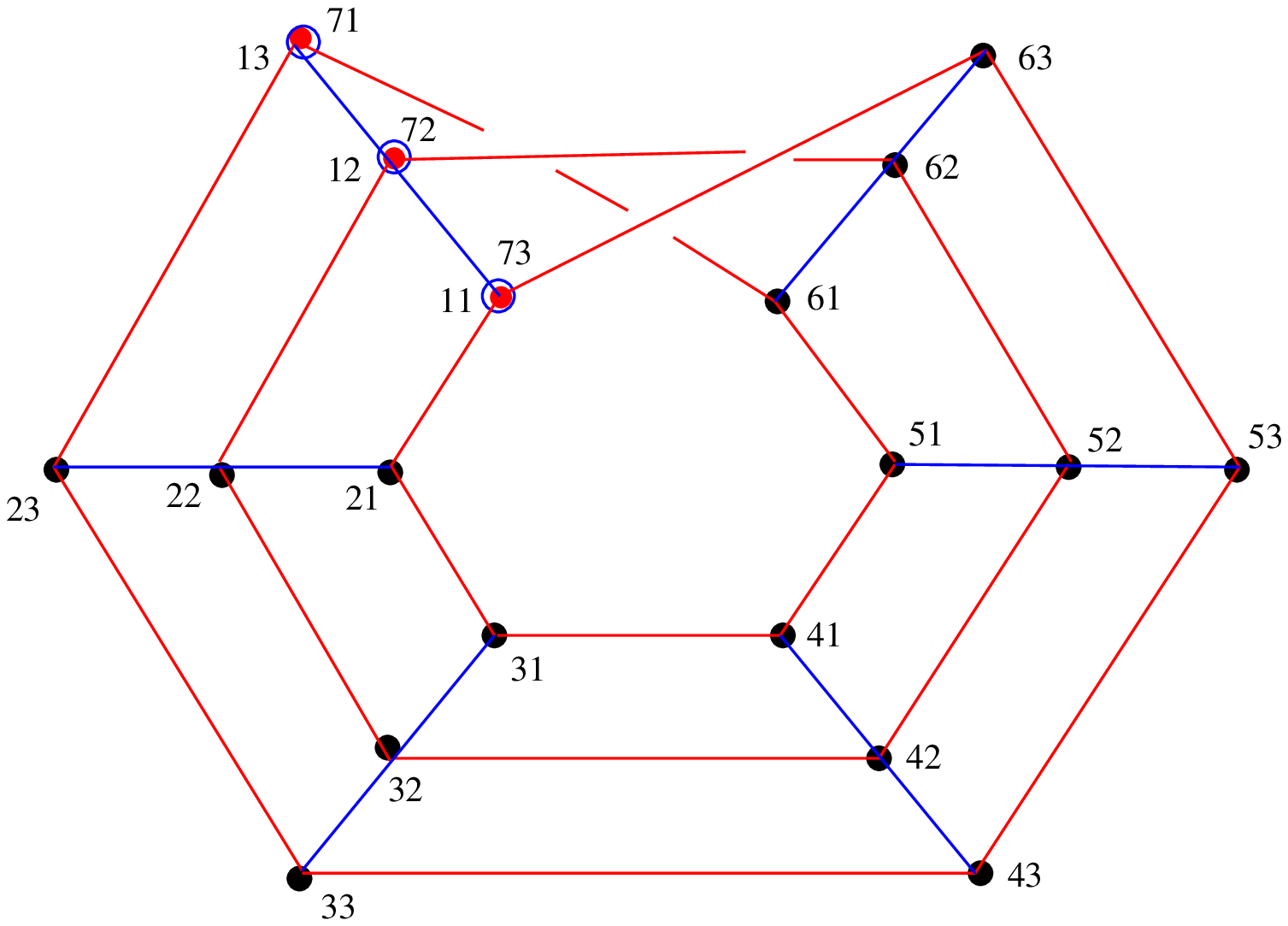, height=4.0cm}
\par Figure 4: The generalized M\"{o}bius ladder $M_{7,3}$
\end{minipage}
\end{center}

\section{Main Results} In this section the general closed formulas of the $M$-polynomial of the generalized M\"{o}bius ladder and its line graph are given.

\begin{thm}\label{thm3.1} The $M$-polynomial of the generalized M\"{o}bius ladder $M_{m,n}, m\geq 4, n\geq 2$, is
$$M(M_{m,n},x,y)=2(m-1)x^{3}y^{3}+2(m-1)x^{3}y^{4}+(m-1)(2n-5)x^{4}y^{4}.$$
\end{thm}
\begin{proof}
Depending on degrees of the vertices, the edges of $M_{m,n}$ can be divided into three disjoint sets: $E_{(3,3)} = \{e=(u,v)\in E\mid \deg(u)=\deg(v)=3\}$, $E_{(3,4)} = \{e=(u,v)\in E\mid \deg(u)=3,\deg(v)=4\}$, and $E_{(4,4)} = \{e=(u,v)\in E\mid \deg(u)=4,\deg(v)=4\}$.\\

\noindent In order to count the number of elements in each of these sets we must consider the grid shape of $M_{m,n}$. The degree-three vertices lie on the top and bottom rows of $M_{m,n}$. Since in each such row there are $m$ vertices, the number of edges whose adjacent vertices are of degree 3 in each such row is $m-1$. Thus, $|E_{(3,3)}|=2(m-1)$.
For better understanding, let us have a look at, for instance, the grid shape of $M_{5,6}$:
\begin{center}
\begin{minipage}{10cm}
\centering  \epsfig{figure=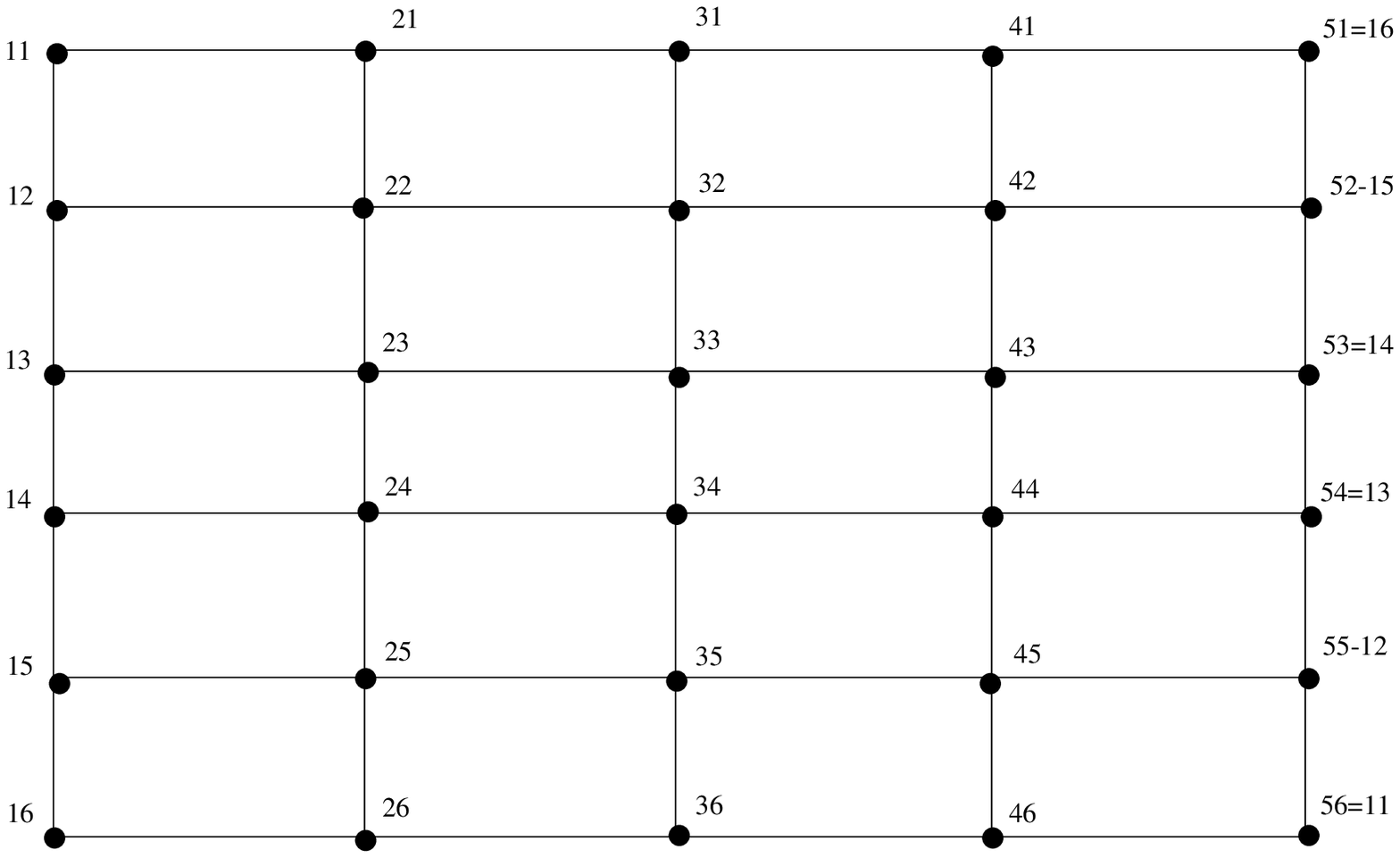, height=6.0cm}
\par Figure 5: Grid of $M_{5,6}$
\end{minipage}
\end{center}
\noindent The degree-three vertices on the top that are adjacent to degree-four vertices determine the edges $(v_{i,1},v_{i,2})$ for $i\in \{1,\ldots,m-1\}$, and the degree-three vertices on the bottom that are adjacent to degree-four vertices determine the edges $(v_{i,n-1},v_{i,n})$ for $i\in \{1,\ldots,m-1\}$; observe that such edges appear vertically. It follows that $m-1$ such edges lie on the top and $m-1$ such edges lie on the bottom of the grid of $M_{m,n}$. Hence $|E_{(3,4)}|=2(m-1)$.\\

\noindent The edges whose adjacent vertices have degrees 4 can be split into two types, horizontal and vertical. The horizontal edges determined by degree-four vertices are $(v_{i,j},v_{i+1,j})$, where for each value of $j$ in $\{2,\ldots,n-1\}$ $i$ takes values in $\{1,\ldots,m-1\}$. Hence, the number of horizontal edges is    $(m-1)(n-2)$. The vertical edges determined by degree-four vertices are $(v_{i,j},v_{i,j+1})$, where for each value of $i$ in $\{1,\ldots,m-1\}$ $j$ takes values in $\{2,\ldots,n-2\}$. Hence, the number of vertical edges is $(m-1)(n-3)$. So, $|E_{(4,4)}|=(m-1)(n-2)+(m-1)(n-3)=(m-1)(2n-5)$, and we are done.
\end{proof}
%%--------------------------------------------Line graph---------------------------------------
\begin{thm}\label{thm3.2}  Let $M_{m,n}$ be the generalized M\"{o}bius ladder for $m,n\geq 4$. Then
$$M(L(M_{m,n}))=2(m-1)x^{4}y^{4}+4(m-1)x^{4}y^{5}+6(m-1)x^{5}y^{6}+6(m-1)(n-3)x^{6}y^{6}.$$
\end{thm}
\begin{proof} Since the line graph $L(M_{m,n})$ of the generalized M\"{o}bius ladder have only vertices of degrees 4, 5, and 6, the edge set of $L(M_{m,n})$ can be divided into five disjoint sets: $E_{(4,4)} = \{e=(u,v)\in E\mid \deg(u)=\deg(v)=4\}$, $E_{(4,5)} = \{e=(u,v)\in E\mid \deg(u)=4,\deg(v)=5\}$, $E_{(5,5)} = \{e=(u,v)\in E\mid \deg(u)=\deg(v)=5\}$, $E_{(5,6)} = \{e=(u,v)\in E\mid \deg(u)=5,\deg(v)=6\}$, and $E_{(6,6)} = \{e=(u,v)\in E\mid \deg(u)=\deg(v)=6\}$.\\

\noindent In order to to count the number of elements in each $E_{(i,j)}$ we need to consider $L(M_{m,n})$ in terms of the grid shape of $M_{m,n}$. The edges of degree 4 appear only on the top and bottom rows of $L(M_{m,n})$. On the top row each degree-four vertex of $L(M_{m,n})$ lies on the edge $(v_{1,j},v_{1,j+1}), 1\leq j\leq m,$ of $M_{m,n}$, and on the bottom row each degree-four vertex of $L(M_{m,n})$ lies on the edge $(v_{n,j},v_{n,j+1}), 1\leq j\leq m,$ of $M_{m,n}$. Since the number of edges determined by degree-four vertices on the top row of $L(M_{m,n})$ is $m-1$ and the number of edges determined by degree-four vertices on the bottom row of $L(M_{m,n})$ is $m-1$, $|E_{(4,4)}|=2(m-1)$. For better understanding, let us have a look at, for instance, the grid of $M_{5,6}$ along with its line graph:
\begin{center}
\begin{minipage}{10cm}
\centering  \epsfig{figure=gridofP5timesP6.eps, height=6.0cm}
\par Figure 6: Grid of $M_{5,6}$
\end{minipage}
\end{center}
\begin{center}
\begin{minipage}{10cm}
\centering  \epsfig{figure=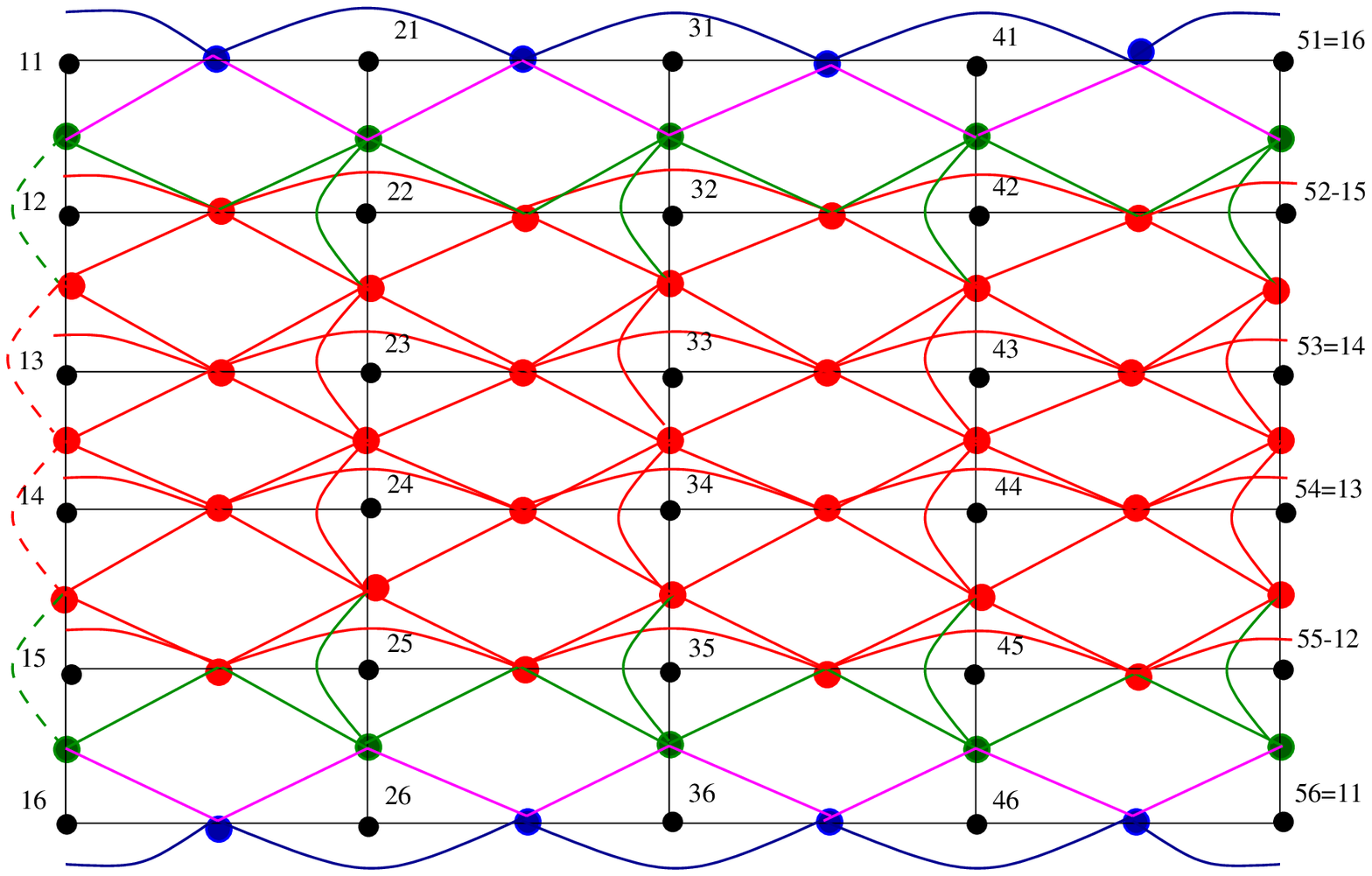, height=7.0cm}
\par Figure 7: Line graph of $M_{5,6}$
\end{minipage}
\end{center}
\noindent Each vertex of degree 4, which is connected to a vertex of degree 5, lies either on the top row or on the bottom row of of $L(M_{m,n})$; see for clarity the line graph of $M_{5,6}$. Each such vertex is connected to two vertices of degree 5; each degree-five vertex of $L(M_{m,n})$ is determined by the edge $(v_{1,j},v_{2,j}), 1\leq j\leq m-1$, on the top and by the edge $(v_{n-1,j},v_{n,j}), 1\leq j\leq m-1$, on the bottom. It follows that the number of edges whose adjacent vertices have degrees 4 and 5 on the top of $L(M_{m,n})$ is $2(m-1)$. Hence, the total number of such edges is $4(m-1)$, i.e., $|E_{(4,5)}|=4(m-1)$.\\

\noindent There is no edge with adjacent vertices of degree 5; to convince yourself, you may see $L(M_{5,6})$.\\

\noindent The degree-five vertices of $L(M_{m,n})$ that are connected to degree-six vertices of $L(M_{m,n})$ lie at the edges $(v_{1,j},v_{2,j})$ and $(v_{n-1,j},v_{n,j})$ for each column $j, 1\leq j \leq m-1$, of $M_{m,n}$. Each of degree-five vertex of $L(M_{m,n})$ that lies at the edge $(v_{1,j},v_{2,j})$ of $M_{m,n}$ is connected with three vertices of degree 6, two are lying at the edges $(v_{1,j},v_{1,j+1})$ and $(v_{1,j+1},v_{1,j+2})$ and one is lying at the edge $(v_{i,j},v_{i+1,j})$ of $M_{m,n}$. Since there are $m-1$ such vertices, the number of edges, whose adjacent vertices are of degree 5 and 6, determined by them is $3(m-1)$. Similarly, the number of edges, whose adjacent vertices are of degree 5 and 6, determined by degree-five vertices of $L(M_{m,n})$ that lie at the edges $(v_{n-1,j},v_{n,j})$ of $M_{m,n}$ is $3(m-1)$. Hence $|E_{(5,6)}|=6(m-1)$.\\

\noindent The edges of $L(M_{m,n})$ whose adjacent vertices have degrees 6 can be divided into three types: the edges that appear horizontally, the edges that appear vertically, and the edges that appear diagonally. Each horizontal edge of $L(M_{m,n})$ with adjacent vertices of degrees 6 appears in row $i, 2\leq i \leq n-1,$ of the grid of $M_{m,n}$. Since in each row there are $m-1$ such edges and there are $n-2$ rows, the total number of horizontal edges in $L(M_{m,n})$ is $(n-2)(m-1)$. The degree-six vertices of $L(M_{m,n})$ that generate vertical edges in $L(M_{m,n})$ are determined by the vertices $(v_{i,j},v_{i+1,j}), 2\leq i \leq n-1, 1\leq j \leq m-1$, of $M_{m,n}$. Since each column of the grid of $L(M_{m,n})$ contains $n-5$ vertices, there are $n-4$ vertical edges in each column. Moreover, since there are $m-1$ columns, the total number of vertical edges in $L(M_{m,n})$ is $(n-4)(m-1)$. The degree-six vertices of $L(M_{m,n})$ that are connected diagonally to degree-six vertices of $L(M_{m,n})$ lie at the edges $(v_{i,j},v_{i+1,j}), 2\leq i \leq n-1$ for each column $j, 1\leq j \leq m-1$, of $M_{m,n}$. Since in each column there are $n-3$ such vertices and each vertex is diagonally connected to 4 vertices of degree 6, the number of edges, whose adjacent vertices are of degree 6, in each column is $4(n-4)$. Moreover, since there are $m-1$ columns, the total number of such edges is $4(n-4)(m-1)$. It now follows that $|E_{(6,6)}|=(n-2)(m-1)+(n-4)(m-1)+4(n-3)(m-1)$. This completes the proof.
\end{proof}
%%--------------------------------------------------Topological Indices----------------------------
\section{Topological Indices}
In this section we recover the first Zagreb, second Zagreb, modified second Zagreb, generalized Randi\'{c}, reciprocal generalized Randi\'{c}, and symmetric division indices from the $M$-polynomials of the generalized M\"{o}bius ladder and its line graph.
\begin{prop}\label{prop4.1}
The degree-based topological indices of $M_{m,n}$ are:
\begin{enumerate}
  \item $M_{1}(M_{m,n})= 16mn - 20m- 16n+14 $
  \item $M_{2}(M_{m,n})=16(4n-3)(n-1)(m-1)^2$
  \item $MM_{2}(M_{m,n})=\frac{1}{144}(6n-1)(6n+1)(m-1)^2$
  \item $R_{\alpha}(M_{m,n})=[16(4n-3)(n-1)(m-1)^2]^\alpha$
  \item $RR_{\alpha}(M_{m,n})= [\frac{1}{144}(6n-1)(6n+1)(m-1)^2]^\alpha$
  \item $SDD(M_{m,n})= \frac{1}{72}(48n^2-42n+1)(m-1)^2$
% \item $H(M_{m,n})= \frac{1}{6}(6n+1)(2n-1)(m-1)^2$
% \item $IS(M_{m,n})= \frac{4}{3}(n-1)(4n-3)(6n+1)(m-1)^3$
% \item $AZI(M_{m,n})= \frac{64}{27}Q_{\alpha}(n-1)^3(4n-3)^3(6n+1)^3(m-1)^9$
\end{enumerate}
\end{prop}
\begin{proof}
From the $M$-polynomial of $M_{m,n}$ we get
\begin{eqnarray*}
  D_x  &=& x \big[\frac{\partial M(M_{m,n})}{\partial x}\big]\\
       &=& x [\frac{\partial }{\partial x}(2(m-1)x^{3}y^{3}+2(m-1)x^{3}y^{4}+(m-1)(2n-5)x^{4}y^{4})]\\
       &=& x [6(m-1)x^{2}y^{3}+6(m-1)x^{2}y^{4}+4(m-1)(2n-5)x^{3}y^{4})]\\
       &=&6(m-1)x^3y^3 +6(m-1)x^3y^4 +4(m-1)(2n-5)x^4y^4.
\end{eqnarray*}
Similarly,
\begin{eqnarray*}
  D_y &=& y \big[\frac{\partial M(M_{m,n})}{\partial y}\big]\\
      &=& 6(m-1)x^3y^3 +8(m-1)x^3y^4 +4(m-1)(2n-5)x^4y^4.
\end{eqnarray*}
Now
\begin{eqnarray*}
  S_x &=& \int_{0}^{x}\frac{M(t,y)}{t}dt\\
      &=& \int_{0}^{x}\frac{1}{t}\big[2(m-1)t^{3}y^{3}+2(m-1)t^{3}y^{4}+(m-1)(2n-5)t^{4}y^{4}\big]dt\\
      &=& (\frac{2m-2}{3})x^3y^3 +(\frac{2m-2}{3})x^3y^4 +\frac{(2n-5)(m-1)}{4}x^4y^4.
\end{eqnarray*}
Similarly,
\begin{eqnarray*}
  S_y &=&  (\frac{2m-2}{3})x^3y^3 +(\frac{2m-2}{4})x^3y^4 +\frac{(2n-5)(m-1)}{4}x^4y^4.
  \end{eqnarray*}
  %\begin{eqnarray*}
%  J &=&  2(m-1)x^{6}+2(m-1)x^{7}+(m-1)(2n-5)x^{8}.
%  \end{eqnarray*}
Finally,
\begin{enumerate}
  \item \begin{eqnarray*}
  M_{1} &=& (D_x+D_y)_{x=y=1} \\
        &=& (8mn-8m-8n+8) + (8mn-12m-8n+6) \\
        &=& 16mn - 20m- 16n+14
\end{eqnarray*}
\item \begin{eqnarray*}
  M_{2} &=& (D_x)_{x=y=1}(D_y)_{x=y=1} \\
        &=& (8mn-8m-8n+8)(8mn-12m-8n+6) \\
        &=& 16(4n-3)(n-1)(m-1)^2
\end{eqnarray*}
\item \begin{eqnarray*}
  MM_{2} &=& (S_x)_{x=y=1}(S_y)_{x=y=1} \\
        &=& (\frac{6mn+m-6n-1}{12})(\frac{6mn- m-6n+1}{12}) \\
        &=& \frac{1}{144}(6n-1)(6n+1)(m-1)^2
\end{eqnarray*}
\item \begin{eqnarray*}
  R_{\alpha}(G) &=& [D_x ^\alpha]_{x=y=1}[D_y^\alpha]_{x=y=1} \\
        &=& [16(4n-3)(n-1)(m-1)^2]^\alpha
\end{eqnarray*}
\item \begin{eqnarray*}
  RR_{\alpha}(G) &=& [S_x ^\alpha]_{x=y=1}[S_y^\alpha]_{x=y=1} \\
        &=& [\frac{1}{144}(6n-1)(6n+1)(m-1)^2]^\alpha
\end{eqnarray*}
\item \begin{eqnarray*}
  SDD &=& (D_xS_y)_{x=y=1} + (S_xD_y)_{x=y=1} \\
        &=& (8mn-8m-8n+8)(\frac{6mn- m-6n+1}{12})\\
        &&+(\frac{6mn+m-6n-1}{12})(8mn-12m-8n+6 ) \\
        &=& \frac{1}{72}(48n^2-42n+1)(m-1)^2
\end{eqnarray*}
%\item \begin{eqnarray*}
%% \nonumber to remove numbering (before each equation)
%      H &=& [2S_x J (M)]_{x=y=1} \\
%        &=& 2(\frac{6mn+m-6n-1}{12})(2mn -m -2n +1)\\
%        &=&  \frac{1}{6}(6n+1)(2n-1)(m-1)^2
%\end{eqnarray*}
%\item \begin{eqnarray*}
%% \nonumber to remove numbering (before each equation)
%      IS &=& [S_x J (D_x D_y )]_{x=y=1} \\
%        &=& (\frac{6mn+m-6n-1}{12})(8mn-8m-8n+8)(8mn-12m-8n+6)\\
%        &=&  \frac{4}{3}(n-1)(4n-3)(6n+1)(m-1)^3
%\end{eqnarray*}
%\item \begin{eqnarray*}
%% \nonumber to remove numbering (before each equation)
%      AZI &=& [S_{x}^3Q_{́\alpha} J (D_{x}^3 D_{y}^3 )]_{x=y=1} \\
%        &=& (\frac{6mn+m-6n-1}{12})^3Q_{́\alpha} (8mn-8m-8n+8)^3(8mn-12m-8n+6)^3\\
%        &=&  \frac{64}{27}Q_{\alpha}(n-1)^3(4n-3)^3(6n+1)^3(m-1)^9
%\end{eqnarray*}
\end{enumerate}
\end{proof}
\begin{prop}\label{prop4.2}
The degree-based topological indices of $L(M_{m,n})$ are:
\begin{enumerate}
  \item $M_{1}(L(M_{m,n}))= 2(36n-49))(m-1)$
  \item $M_{2}(L(M_{m,n}))=72(9n-11))(2n-3)(m-1)^2$
  \item $MM_{2}(L(M_{m,n}))=\frac{1}{100} (10n-3)(10n-7)(m-1)^2$
  \item $R_{\alpha}(L(M_{m,n}))=[72(9n-11))(2n-3)(m-1)^2]^\alpha$
  \item $RR_{\alpha}(L(M_{m,n}))= [\frac{1}{100} (10n-3)(10n-7)(m-1)^2]^\alpha$
  \item $SDD(L(M_{m,n}))= \frac{1}{72}(48n^2-42n+1)(m-1)^2$
 %\item $H(L(M_{m,n}))= \frac{6}{5} (10n-3)(n-1)(m-1)^2$
% \item $IS(L(M_{m,n}))= \frac{36}{5}(2n-3)(10n-3)(9n-11)(m-1)^3$
% \item $AZI(L(M_{m,n}))= \frac{46656}{125}Q_{\alpha}(2n-3)^3(10n-3)^3(9n-11)^3(m-1)^9$
\end{enumerate}
\end{prop}
\begin{proof}
The proof is similar to the proof of Proposition~\ref{prop4.1}.
\end{proof}


\begin{thebibliography}{99}

\bibitem{ajmal:17}{M. Ajmal, W. Nazeer, M. Munir,  S.M. Kang and C.Y. Jung, M-polynomials and topological indices of generalized prism network, \em International Journal of Mathematical Analysis, }{\bf 11} (2017), 293 - 303.

\bibitem{Amic-Beslo:98} {D. Amic, D. Beslo, B. Lucic, S. Nikolic and N. Trinajsti\'{c}, The Vertex-Connectivity Index Revisited, \em J. Chem. Inf. Comput. Sci.,} {\bf 38} (1998), 819-822.

\bibitem{Bollobas-Erdos:98} {B. Bollobas and P. Erdos, Graphs of extremal weights, \em Ars Combinnatoria,} {\bf 50} (1998), 225-233.

\bibitem{Braun-Kerber-Meringer-Rucker:05} {J. Braun, A. Kerber, M. Meringer, and C. Rucker, Similarity of Molecular Descriptors: the Equivalence of Zagreb Indices and Walk Counts, \em MATCH Communications in Mathematical and in Computer Chemistry,} {\bf 54} (2005), 163–176.

\bibitem{Das-Xu-Nam:15} {K.C. Das, K. Xu, J. Nam, Zagreb Indices of Graphs, \em Front. Math. China,} {\bf 10} (2015), 567–582.

\bibitem{Deutsch:15}{E. Deutsch and S. Kav\v{z}ar, M-polynomial and degree-based topological indices, \em Iran. J. Math. Chem., }{\bf 6} (2015), 93-102.

\bibitem{Fajtlowicz:87} {S. Fajtlowicz, On Conjectures of Graffiti—II, \em Congr. Numer.,} {\bf 60} (1987), 187–197.

\bibitem{Furtula-Graovac-Vukicevic:10} {B. Furtula, A. Graovac, and D. Vukicevic, Augmented Zagreb Index, \em J. Math. Chem.,} {\bf 48} (2010), 370–380.

\bibitem{Gutman:13} {I. Gutman, Degree-Based Topological Indices, \em Croat. Chem. Acta,} {\bf 86} (2013), 351-361.

\bibitem{Gutman-Das:04} {I. Gutman, and K.C. Das, The First Zagreb Index 30 Years After, \em MATCH Communications in Mathematical and in Computer Chemistry,} {\bf 50} (2004), 83–92.

\bibitem{Gutman-Furtula:08} {I. Gutman and B. Furtula, \em Recent Results in the Theory of Randic Index,}  MCM, Kragujevac, 2008.

\bibitem{Gutman-Polansky:08}  {I. Gutman and O.E. Polansky, \em Mathematical Concepts in
Organic Chemistry,} Springer-Verlag, Berlin, 1986.

\bibitem{Gutman-Tosovic:13} {I. Gutman, J. To\v{s}ovi\'{c}, Testing the Quality of Molecular Structure Discriptors: Vertex Degree Based Topological Indices \em J. Serb. Chem. Soc.,} {\bf 78} (2013), 805-810.

\bibitem{Gutman-Trinajstic:72} {I. Gutman, N. Trinajstic, Graph theory and molecular orbitals, , Total pelectron energy of alternant hydrocarbons \em Chem. Phys. Lett.,} {\bf 17} (1972), 535–538.

\bibitem{Khalifeha-Yousefi-Ashrafi:09} {M.H. Khalifeha, Yousefi-Azaria, and A.R. Ashrafi, The First and Second Zagreb Indices of Some Graph Operations, \em Discrete Appl. Math.,} {\bf 157} (2009), 804-811.

\bibitem{Li-Gutman:06} {X. Li, I. Gutman, \em Mathematical Aspects of Randic-type Descriptors,}  MCM, Kragujevac, 2006.
\bibitem{Li-Gutman2:06} {X. Li, I. Gutman, \em Mathematical Chemistry Monographs,} Kragujevac, 2006.

\bibitem{Hongbin-Idrees-Nizami-Munir:17} {M. Hongbin, M. Idrees, A.R. Nizami, and M. Munir, Generalized M\"{o}bius Ladder and Its Metric Dimension, \em arXiv:1708.05199v1}.

\bibitem{Huang-Liu-Gan:12} {Y. Huang,  B. Liu, and L. Gan, Augmented Zagreb Index of Connected Graphs, \em MATCH Communications in Mathematical and in Computer Chemistry,} {\bf 67} (2012), 483-494.

\bibitem{Hu-Li-Shi-Xu-Gutman:05} {Y. Hu,  X. Li, Y. Shi, T. Xu and I. Gutman, On molecular graphs with smallest and greatest zeroth-Corder general randic index, \em MATCH Communications in Mathematical and in Computer Chemistry,} {\bf 54} (2005), 425-434.

\bibitem{idrees-hongbin-nizami-mobeen:17} {M. Idrees,  M. Hongbin, A.R. Nizami, and M. Munir, Generalized M\"{o}bius Ladder and Its Metric Dimension, \em arXiv.}

\bibitem{mobeen:17}{M. Munir, W. Nazeer, A.R. Nizami, S. Rafique and S.M. Kang, M-polynomials and topological indices of titania nanotubes, \em Syemmetry, }{\bf 8} (2017), 1-9.

\bibitem{mobeen-waqas-nizami2:17} {M. Munir, W, Nazeer, S. Rafique, A.R. Nizami, and S.M. Kang, Some Computational Aspects of Triangular Boron Nanotubes, \em Symmetry,} {\bf 9} (2017), 1-11.

\bibitem{mobeen-waqas-nizami3:17} {M. Munir, W, Nazeer, S.M. Kang, M.I. Qureshi, A.R. Nizami, and Y.C. Kwun, Some Invariants of Jahangir Graphs, \em Symmetry,} {\bf 9} (2017), 1-15.

\bibitem{Nikolic-Kovacevic-Milicevic-Trinajstic:08} {S. Nikolic, G. Kovacevic, A. Milicevic, and N. Trinajstic, The Zagreb Indices 30 Years After, \em Croat. Chem. Acta,} {\bf 76} (2003), 113–124.

\bibitem{Todeschini-Consonni:00} {R. Todeschini, V. Consonni, \em Handbook of Molecular Descriptors,} Wiley-VCH, Weinheim, 2000.

\bibitem{Randic:75} {M. Randic,  Characterization of molecular branching, \em Journal of the American Chemical Society,} {\bf 97} (1975), 6609–6615.

\bibitem{Xu-Tang-Liu-Wang:15} {K. Xu, K. Tang, H. Liu, and J. Wang, The Zagreb Indices of Bipartite Graphs with More Edges, \em J. Appl. Math. Inf.,} {\bf 33} (2015), 365-377.

\bibitem{Xueliang-Yongtang:08} {Li. Xueliang and Shi. Yongtang, A survey on the Randic index, \em MATCH Communications in Mathematical and in Computer Chemistry,} {\bf 59} (2008), 127–156.

\bibitem{Zhou-Gutman:04} {B. Zhou, and I. Gutman, Relations Between Wiener, Hyper-Wiener And Zagreb Indices, \em Chem. Phys. Lett.,} {\bf 394} (2004), 93–95.

\bibitem{Zhou-Gutman2:05} {B. Zhou, and I. Gutman, Further Properties of Zagreb Indices, \em MATCH Commun. Math. Comput. Chem.,} {\bf 54} (2005), 233–239.

\bibitem{Zhou:07} {B. Zhou, Upper Bounds for the Zagreb Indices and the Spectral Radius of Series-Parallel Graphs, \em Int. J. Quantum Chem.,} {\bf 107} (2007), 875-878.

\bibitem{Zhou2:07} {B. Zhou, Remarks on Zagreb Indices, \em MATCH Commun. Math. Comput. Chem.,} {\bf 57} (2007), 591-596.
\end{thebibliography}
\end{document}